\documentclass[a4paper,12pt]{amsart}

\usepackage{amsfonts}
\usepackage{amsmath}
\usepackage{amsthm}
\usepackage{amssymb}
\usepackage{mathtools}
\usepackage{hyperref}
\usepackage{wrapfig}
\usepackage{tikz}

\newcommand{\diam}{\mathrm{diam}}
\newcommand{\Kinfty}{\mathcal{K}_\infty}

\newcommand{\grad}{\mathrm{grad}}
\newcommand{\Ric}{\mathrm{Ric}}

\newcommand{\Id}{\mathrm{Id}}
\newcommand{\NN}{\mathbb{N}}
\newcommand{\KK}{\mathcal{K}}
\newcommand{\RR}{\mathbb{R}}
\newcommand{\ZZ}{\mathbb{Z}}
\newcommand{\complete}{{\rm cmp}}
\newcommand{\incomplete}{{\rm inc}}

\newtheorem{theorem}{Theorem}[section]

\newtheorem{corollary}[theorem]{Corollary}
\newtheorem{proposition}[theorem]{Proposition}
\newtheorem{definition}[theorem]{Definition}
\newtheorem{remark}[theorem]{Remark}
\newtheorem*{conj}{Conjecture}

\makeatletter
\renewcommand\subsubsection{\@startsection{subsubsection}{2}%
  \z@{.5\linespacing\@plus.7\linespacing}{-.5em}%
  {\normalfont\bfseries}}
\makeatother

\begin{document}

\title{Quartic graphs which are Bakry-\'Emery curvature sharp}


\author[Cushing]{David Cushing}
\address{D. Cushing, Department of Mathematical Sciences, Durham University, Durham DH1 3LE, United Kingdom}
\email{davidcushing1024@gmail.com}

\author[Kamtue]{Supanat Kamtue}
\address{S. Kamtue, Department of Mathematical Sciences, Durham University, Durham DH1 3LE, United Kingdom}
\email{supanat.kamtue@durham.ac.uk}

\author[Peyerimhoff]{Norbert Peyerimhoff}
\address{N. Peyerimhoff, Department of Mathematical Sciences, Durham University, Durham DH1 3LE, United Kingdom}
\email{norbert.peyerimhoff@durham.ac.uk}

\author[Watson May]{Leyna Watson May}
\address{L. Watson May, Department of Mathematical Sciences, Durham University, Durham DH1 3LE, United Kingdom}
\email{leyna.may@durham.ac.uk}

\maketitle

\begin{abstract}
  We give a classification of all connected quartic graphs which are
  (infinity) curvature sharp in all vertices with respect to
  Bakry-\'Emery curvature. The result is based on a computer
  classification by F. Gurr and L. Watson May and a combinatorial case
  by case investigation.
\end{abstract}

\section{Introduction}
Curvature is a fundamental notion in geometry which goes back to Gauss
and Riemann and was originally defined in the smooth setting of
Riemannian manifolds. A challenging problem is to find meaningful
curvature notions in discrete settings like graphs and networks.

In this paper we focus on a specific curvature notion on a graph
$G=(V,E)$ with a vertex set $V$ and an edge set $E$, called
Bakry-\'Emery curvature (at dimension $n=\infty$). This curvature
notion is based on Bakry-\'Emery's $\Gamma$-calculus and a curvature
dimension inequality \cite{BE85}, and it was first used by
Schmuckenschl\"ager \cite{Schm99} in 1999. The crucial ingredient to
define this curvature is a natural notion of a Laplacian. In this
paper, we choose the \emph{non-normalized} graph Laplacian $\Delta$, defined
on functions $f: V \to \RR$ by
$$ \Delta f(x) = \sum_{y: y \sim x} (f(y)-f(x)). $$
Bakry-\'Emery curvature is then a real valued function
$\mathcal K_\infty(x)$ of the vertices $x\in V$, where the value
$\mathcal K_\infty(x)$ is fully determined by the combinatorial
structure of the (incomplete) $2$-ball around $x$. The precise definition of
$\mathcal K_\infty(x)$ is given in Subsection
\ref{sect:Be_curv-mot}. For readers interested in more details about
this curvature notion, see, e.g. \cite{CLP17} and the references
therein. Corollary 3.3 of \cite{CLP17} gives the following upper bound
of $\mathcal K_\infty(x)$ for a $D$-regular graph:

\begin{equation*} \label{eq:upcurvbd0}
\mathcal K_\infty(x) \le 2 + \frac{\#_\Delta(x)}{D},
\end{equation*}
where $\#_\Delta(x)$ is the number of triangles containing $x$ as a
vertex. We call a vertex $x$ \emph{(infinity) curvature sharp} if this
estimate holds with equality. The main result in the paper is a complete
classification of all $4$-regular (quartic) curvature sharp graphs:

\begin{theorem} \label{thm:main}
	Let $G=(V,E)$ be a connected quartic graph which is Bakry-\'Emery
	curvature sharp in all vertices. Then $G$ is one of the following:
	\begin{itemize}
		\item[(i)] The complete graph $K_5$ with $|V|=5$, $\Kinfty=3.5$,
		$\diam\, G = 1$;
		\item[(ii)] The octahedral graph $O$ with $|V|=6$, $\Kinfty=3$,
		$\diam\, G = 2$;
		\item[(iii)] The Cartesian product $K_3 \times K_3$ of two copies of
		the complete graph $K_3$ with $|V|=9$, $\Kinfty=2.5$, $\diam\, G =
		2$;
		\item[(iv)] The complete bipartite graph $K_{4,4}$ with $|V|=8$,
		$\Kinfty=2$, $\diam\, G = 2$;
		\item[(v)] The crown graph $C(10)$ with $|V|=10$, $\Kinfty=2$,
		$\diam\, G = 3$;
		\item[(vi)] The Cayley graph $Cay(D_{12},S)$ of the dihedral group
		$D_{12}$ of order $12$ with generators $S = \{ r^3,s,sr^2,sr^4 \}$
		with $|V|=12$, $\Kinfty=2$, $\diam\, G = 3$;
		\item[(vii)] The Cayley graph $Cay(D_{14},S)$ of the dihedral group
		$D_{14}$ of order $14$ with generators $S = \{ s,sr,sr^4,sr^6 \}$
		with $|V|=14$, $\Kinfty=2$, $\diam\, G = 3$;
		\item[(viii)] The $4$-dimensional hypercube $Q^4$ with $|V|=16$,
		$\Kinfty=2$, $\diam\, G = 4$.
	\end{itemize}
\end{theorem}

The proof of this result is based on a computer classification of all
quartic incomplete $2$-balls with non-negative curvature at their
centres. This local classification result was obtained in a 2018
LMS\footnote{London Mathematical Society} Undergraduate Research
Bursary by the last author. The revelant local results of this
research can be found in \cite{GW18} and are summarized in Section
\ref{subsec:python_results} below. They are crucial for the
combinatorial arguments given in Section \ref{sect:proof_main} to
derive the global classification result. In fact, the proof of Theorem
\ref{thm:main} is a combinatorial case by case investigation starting
with an incomplete $2$-ball with a curvature sharp center and
extending it to derive a contradiction or to end up with one of the graphs in
the above classification.

It is conceivable that the results in \cite{GW18} may
also have other applications, for example with regards to the following
conjecture about expander graph families (Conjecture 9.11 in \cite{CLP17}):

\begin{conj} Let $D \in \NN$ be fixed. Then there do {\bf{not}} exist
  increasing $D$-regular expander graphs $\{ G_k \}_{k \in \NN}$ which are
  non-negatively curved in all vertices.
\end{conj}  

In the case $D=3$ (cubic graphs), it was shown in
\cite[Theorem~1.1]{CKLLS17} that the only finite connected cubic
graphs of non-negative Bakry-\'Emery curvature are the prism graphs
and the M\"obius ladders and, therefore, the conjecture is true for
$D=3$. It would be an interesting project to investigate whether the
results in \cite{GW18} can be used to verify the conjecture in the case
$D=4$. A full classification of all finite connected non-negatively
curved quartic graphs is likely to be out of range due to the large
number of local combinatorial possibilities to construct such
graphs. However, \cite{GW18} might be useful to derive specific
properties contradicting the existence of expander families like,
e.g., polynomial volume growth of metric balls.

Let us finish this introduction with an overview about the structure
of this paper. In Section \ref{sect: BE_curv}, we introduce
Bakry-\'Emery curvature and all other relevant notions and present
some crucial results needed for the proof of Theorem
\ref{thm:main}. The proof of Theorem \ref{thm:main} is given in Section
\ref{sect:proof_main}.

\section{Bakry-\'Emery curvature} \label{sect: BE_curv}

\subsection{Motivation of Bakry-\'Emery curvature} \label{sect:Be_curv-mot}
Readers not familiar with Riemannian
manifolds can skip the following explanation and go directly to Definition
\ref{def:curv} below.

Bakry-\'Emery
curvature is a general (Ricci) curvature notion which can be motivated
via the following \emph{curvature-dimension inequality} on an
$n$-dimensional Riemannian manifold $(M,\langle \cdot,\cdot \rangle)$
whose Ricci curvature at $x$ satisfies $\Ric_x(v) \ge K_x |v|^2$ for all
tangent vectors $v \in T_xM$:
$$
\frac{1}{2} \Delta \vert \grad\, f|^2 (x) - \langle \grad\, \Delta
f(x), \grad\, f(x) \rangle \ge \frac{1}{n}(\Delta f(x))^2 + K_x
\Vert \grad\, f(x) \Vert^2.
$$
This pointwise inequality holds for all smooth functions $f$ and is a
straightforward consequence of \emph{Bochner's formula}, a fundamental
fact in Riemannian Geometry (for Bochner's formula see, e.g.,
\cite{GHL04}). Using Bakry-\'Emery's $\Gamma$-calculus, this
inequality can be reformulated as
\begin{equation} \label{eq:cd-ineq}
\Gamma_2(f,f)(x) \ge \frac{1}{n}(\Delta f(x))^2 + K_x
\Gamma(f,f)(x) \quad \forall\, f,
\end{equation}
where the symmetric bilinear forms $\Gamma$ and $\Gamma_2$ of two
smooth function $f,g: M \to \RR$ are defined as
\begin{eqnarray}
  2 \Gamma(f,g) &=& \Delta(fg) - f\Delta g - g\Delta f = \langle \grad\,
  f,\grad\, g \rangle, \label{eq:gamma} \\
  2 \Gamma_2(f,g) &=& \Delta \Gamma(g,f) - \Gamma(f,\Delta g) -
  \Gamma(g,\Delta f). \label{eq:gamma2}
\end{eqnarray}
Note that $\Gamma$ and $\Gamma_2$ can be defined for any space
admitting a reasonable Laplace operator $\Delta$. The idea is to use
inequality \eqref{eq:cd-ineq} to define lower Ricci curvature bounds
at all points of general spaces admitting Laplace operators. In the
case of an arbitrary (not necessarily regular) graph $G=(V,E)$ with
vertex set $V$ and edge set $E$, there is a natural way to introduce a
Laplace operator via its adjacency matrix $A_G$, namely
$$ \Delta = A_G - D\cdot \Id,$$
where $D$ is a diagonal matrix containing the respective vertex
degrees. This operator $\Delta$ is called the \emph{non-normalized
  graph Laplacian} and can also be viewed as a linear operator on the
space of functions on the vertices. It is straightforward to see that
the Laplacian of a function $f: V \to \RR$ is then given by
\begin{equation} \label{eq:norm-Lap}
\Delta f(x) = \sum_{y: y \sim x} (f(y)-f(x)),
\end{equation}
where $y \sim x$ means that the vertices $x$ and $y$ are adjacent.

Note that inequality \eqref{eq:cd-ineq} involves a dimension parameter
$n$, and it is not clear how to choose the dimension for a given graph
$G$. If we do not fix the dimension parameter $n$, \eqref{eq:cd-ineq}
induces a lower Ricci curvature notion at a vertex $x\in V$ \emph{as a
  function of the dimension}. This viewpoing was taken in \cite{CLP17}
and it easy to see that this pointwise curvature function is
monotone increasing in $n$ and assumes a finite limit as $n \to
\infty$. We refer to the limit as the Bakry-\'Emery curvature (at infinity)
$\mathcal K_\infty(x)$ at the vertex $x$. This limit value can also be
directly obtained by dropping the term involving the dimension
parameter in \eqref{eq:cd-ineq}.

\begin{definition} \label{def:curv}
Let $G=(V,E)$ be a graph and $\Delta$ be the associated Laplacian defined in \eqref{eq:norm-Lap}. Let $\Gamma$ and $\Gamma_2$ be the forms defined in \eqref{eq:gamma} and \eqref{eq:gamma2}. Then the Bakry-\'Emery curvature
$\mathcal K_\infty(x)$ at a vertex $x$ is the supremum of all values $K\in \RR$ satisfying
$$ \Gamma_2(f,f)(x) \ge K \Gamma(f,f)(x) \quad \forall\, f: V \to \RR. $$ 
Moreover, if we have $\mathcal K_\infty(x) \ge \KK$ at all vertices $x \in V$ for some value $\KK \in \RR$, we say that $G$ satisfies the (global) curvature-dimension inequality $CD(\KK,\infty)$.
\end{definition}

A natural class of connected regular graphs satisfying $CD(0,\infty)$
are all abelian Cayley graphs (see \cite{KKRT16} and references therein) and a
prominent example with vanishing Bakry-\'Emery curvature at all
vertices is the infinite grid $\ZZ^n$ with generators $\pm e_j$,
$j=1,\dots,n$.

\subsection{Fundamental properties of Bakry-\'Emery curvature}

Before we present some fundamental properties of Bakry-\'Emery
curvature, we need to introduce some relevant notation. All graphs $G =
(V,E)$ are assumed to be \emph{connected}, i.e., there is a path between any pair of vertices in $V$. The degree of a vertex $x$ is denoted by $d_x\in \mathbb{N}$, and a graph $G$ is called \emph{$D$-regular} if $d_x=D$ for all $x\in V$. The \emph{combinatorial distance} $d(x,y)$ between
two vertices $x,y \in V$ is then the length of the shortest path from
$x$ to $y$. The \emph{diameter} of $G$ is defined as $$\diam(G)=\max_{x,y\in V} d(x,y).$$
Spheres and balls around a vertex $x \in
V$ are defined via
\begin{eqnarray*}
  S_k(x) &=& \{ y \in V \mid d(x,y) = k \}, \\
  B_k(x) &=& \{ y \in V \mid d(x,y) \le k \}.
\end{eqnarray*}
The $2$-ball $B_2(x)$ has the following decomposition into spheres
$$ B_2(x) = \{x\} \sqcup S_1(x) \sqcup S_2(x). $$
We call an edge $\{y,z\} \in E$ a \emph{spherical edge} (w.r.t. $x$)
if $d(x,y) = d(x,z)$, and a \emph{radial edge} otherwise. Moreover,
the following values associated a reference vertex $x \in V$ are
relevant:
\begin{align*}
d_{x}^{-}(y) &= |\{z\sim y : d(x,y) = d(x,z) + 1\}|,\\
d_{x}^{0}(y) &= |\{z\sim y : d(x,y) = d(x,z) \}|,\\
d_{x}^{+}(y) &= |\{z\sim y : d(x,y) = d(x,z) - 1\}|,
\end{align*}
which we call the {\it in-degree, spherical degree, out-degree} of
$y$, respectively. Note that $d_y=d_{x}^{-}(y)+d_{x}^{0}(y)+d_{x}^{+}(y)$.

\begin{definition}
  We say that $G$ is \emph{$S_1$-out regular} at a vertex $x$, if all
  the vertices $y$ in $S_1(x)$ have the same out-degree
  $d^{+}_{x}(y)$.
\end{definition}
  
The \emph{complete 2-ball} around $x$, denoted by $B_2^\complete(x)$, is the induced subgraph of $B_2(x)$. Furthermore, the \emph{incomplete 2-ball} around $x$, denoted by $B_2^\incomplete(x)$, is obtained from $B_2^\complete(x)$ with all spherical edges w.r.t. $x$ within $S_2(x)$ being removed. It is important to note that Bakry-\'Emery curvature $\mathcal K_\infty(x)$ at a vertex $x \in V$ is a local value, and it is already determined by the structure of incomplete $2$-ball $B_2^\incomplete(x)$. As explained in \cite[Section 3.4]{CKLLS17}, the explicit calculation of Bakry-\'Emery curvature at a vertex is a semidefinite programming problem implemented in the interactive curvature calculator which can be found at \url{http://www.mas.ncl.ac.uk/graph-curvature}.

In \cite{CLP17}, the authors give an upper bound for Bakry-\'Emery curvature at a vertex in a $D$-regular graph, and then define the notion of a curvature sharp vertex as follows:

\begin{theorem}\textup{(\cite[Corollary~3.3]{CLP17})}
Let $G=(V,E)$ be a $D$-regular graph. Then Bakry-\'Emery curvature (of dimension $n=\infty$) at any vertex $x \in V$ satisfies
\begin{equation} \label{eq:upcurvbd}
\mathcal K_\infty(x) \le 2 + \frac{\#_\Delta(x)}{D}.
\end{equation}

Moreover, a vertex $x \in V$ is called
(infinity) curvature sharp if \eqref{eq:upcurvbd} holds with
equality.
\end{theorem}

From Corollary 5.11 in \cite{CLP17}, curvature sharp at $x$ implies $S_1$-out regularity at $x$. Moreover, the following proposition says that if $S_1$-out regularity is assumed everywhere, then the number of triangles involving a vertex (or an edge) is uniform.

\begin{proposition} \label{prop: number_triangles}
Let $G=(V,E)$ be a connected $D$-regular graph which is $S_1$-out regular in all vertices. Then there is $c_1,c_2\in \ZZ$ such that $\#_\Delta(\{x,y\})=c_1$ and $\#_\Delta(x)=c_2$, for all $x\in V$ and for all edges $\{x,y\}\in E$.
\end{proposition}

\begin{proof}
Let $x\sim y \sim z$ be any two connected edges in $G$. Observe that
\begin{align*}
d^+_y(x)&=D-1-\#_\Delta(\{x,y\})\\
d^+_y(z)&=D-1-\#_\Delta(\{y,z\}).
\end{align*}
Then $S_1$-out regularity at $y$ means $d^+_y(x)=d^+_y(z)$, which implies that  $\#_\Delta(\{x,y\})=\#_\Delta(\{y,z\})$.

Let $x\sim y$ and $x'\sim y'$ be two arbitrary edges in $G$. Consider a connected path from the edge $x\sim y$ to the edge $x'\sim y'$, namely $$x\sim y \sim v_0 \sim ... \sim v_n \sim x' \sim y'.$$ 
Then $S_1$-out regularity at $y$, at $v_i$'s, and at $x'$ altogether implies that $$\#_\Delta(\{x,y\})= \#_\Delta(\{y,v_0\})=...=\#_\Delta(\{v_n,x'\})=\#_\Delta(\{x',y'\}),$$ so $\#_\Delta(\{x,y\})\equiv c_1$, for some constant $c_1$. Moreover, counting triangles containing $x$ gives $$\#_\Delta(x)=\frac{1}{2} \sum_{y:\ y\sim x}\#_\Delta(\{x,y\})=\frac{1}{2}c_1 D,$$
which is also a constant i.e., $c_2:=\frac{1}{2}c_1 D$.
\end{proof}

\begin{corollary}
  Every connected $D$-regular graph $G = (V,E)$ which is curvature
  sharp in all vertices $x$ has constant curvature, that is, $\mathcal K_\infty(x) = \mathcal K$ for all $x \in V$.
\end{corollary}
\begin{proof}
	$\#_\Delta(x)$ is constant for all vertices $x\in V$ by Proposition \ref{prop: number_triangles}, and the result follows from \eqref{eq:upcurvbd}.
\end{proof}

Finally, we need the following \emph{combinatorial analogue} of the classical
Bonnet-Myers Theorem from Riemannian Geometry (see, e.g.,
\cite{GHL04}) and its associated rigidity result by Cheng \cite{Ch75}.

\begin{theorem}\label{thm:BMrigidity}\textup{(\cite[Proposition~1.3 and Theorem~1.4]{rigidity})} Let $G=(V,E)$ be a connected $D$-regular graph with $K:=\inf_{x\in V} \mathcal K_\infty(x)>0$. Then $G$ satisfies Bonnet-Myers' diameter bound
$$\diam(G)\le \frac{2D}{K},$$ which holds with equality if and only if $G$ is a $D$-dimensional hypercube. 
\end{theorem}

\subsection{Incomplete $2$-balls with non-negative curvature at centers}
\label{subsec:python_results}

In this subsection, we survey the relevant computational results about
Bakry-\'Emery curvature from \cite{GW18}. They are based on a computer
program in Python written by the last author during a 2018 LMS
Undergraduate Research Bursary.

Firstly, we explain the representations of $2$-balls in quartic graphs that were used for these calculations. We fix a vertex $v_0$ and $S_1(v_0)=\{v_1,v_2,v_3,v_4\}$ (since in a quartic graph, $v_0$ has four neighbors). The vertices of the $2$-sphere are labeled as follows: $S_2(v_0)=\{v_5,v_6,...,v_m\}$. Then a $2$-ball $B_2(v_0)$ centered at $v_0$ is represented by a list of 3 lists: $B_2(v_0)=[list_1,list_2, list_3]$. The first list determines the $S_1$ structure (i.e., how the vertices in $S_1$ are connected to each other) by $list_1=[a_{12},a_{13},a_{14},a_{23},a_{24},a_{34}]$ where each $a_{ij}\in\{0,1\}$ is a Boolean indicator whether vertices $v_i$ and $v_j$ are adjacent or not. The second list $list_2=[a_5,a_6,...,a_m]$ describes the $S_1\mbox{-}S_2$ structure (i.e., which vertices in $S_2$ are adjacent to vertices in $S_1$). For instance, $a_5=[123]$ means that the vertex $v_5$ is adjacent to $v_1,v_2,v_3$ but not to $v_4$. Lastly, the list $list_3$ describes the $S_2$ structure (i.e., how the vertices in $S_2$ are connected to each other). For example, $list_3=[[57], [58], [68]]$ means that $v_5\sim v_7$, $v_5\sim v_8$, and $v_6\sim v_8$. However, the computation of Bakry-\'Emery curvature only requires the information of incomplete $2$-balls where no spherical edge of $S_2(v_0)$ is present, in which case $list_3=[\ ]$.
We refer to quartic incomplete $2$-balls as those which are incomplete $2$-balls of some quartic graph, i.e., every quartic $B_2^\incomplete(v_0)$ has $d_{v_0}=d_{v_1}=...=d_{v_4}=4$ and $d_{v_5},...,d_{v_m} \le 4$.

For example, the incomplete $2$-ball $B_2^\incomplete(v_0)$ in Figure \ref{fig: ex_2ball} has the following representation:
\begin{align*}
B_2^\incomplete(v_0)=\bigg[\underbrace{\big[0,1,0,0,1,0\big]}_{\text{$S_1$ structure}}\ ,\ \underbrace{\big[[13],[13],[24],[2],[4]\big]}_{\text{$S_1$-$S_2$ structure}}\ ,\ \big[\ \big]\bigg]
\end{align*}

\begin{figure}[h!]
\begin{center}
\begin{tikzpicture}[scale=1.5]
	\tikzstyle{every node}=[draw, shape=circle, scale = 0.7, thick]
        \path(1:0cm)	node(v0) [text = black]{$v_0$};
	\path(0:1cm)	node(v1) [red, text =black]{$v_1$};
        \path(90:1cm)	node(v2) [red, text =black]{$v_2$};
	\path(180:1cm)	node(v4) [red, text =black]{$v_4$};
        \path(270:1cm)	node(v3) [red, text =black]{$v_3$};
	
	\path(315:2cm)	node(v5) [blue, text =black]{$v_5$};
	\path(135:2cm)	node(v6) [blue, text =black]{$v_6$};
	\path(180:2cm)	node(v7) [blue, text =black]{$v_7$};
	\path(90:2cm)	node(v8) [blue, text =black]{$v_8$};
	\path(240:2cm)	node(v9) [blue, text =black]{$v_9$};		
	\draw(v0) -- (v1)
	(v0) -- (v2)
	(v0) -- (v3)
	(v0) -- (v4)
	(v2) edge[red] (v4)
	(v1) edge[red] (v3)
	(v1) edge[blue] (v5) 
	(v3) edge[blue] (v5)
	(v1) edge[blue] (v6) 
	(v3) edge[blue] (v6)
	(v2) edge[blue] (v7) 
	(v4) edge[blue] (v7)
	(v2) edge[blue] (v8)
	(v4) edge[blue] (v9);	
\end{tikzpicture}
\caption{$B_2^\incomplete(v_0)=\Big[\big[0,1,0,0,1,0\big],\big[[13],[13],[24],[2],[4]\big],\big[\ \big]\Big]$}
\label{fig: ex_2ball}
\end{center}
\end{figure}

Concerning the $S_1$ structures, we will only consider the 11 standard
representations (given in the second column of Table
\ref{table:count_2ball}) since all other $S_1$ structures can be
obtained from them via permutations of the vertices $v_1,v_2,v_3,v_4$.

The computational results are presented in the following two propositions below. The first proposition gives the number of all non-isomorphic quartic incomplete $2$-balls as well as the ones with non-negative curvature at their center. The second proposition gives a list of all 22 quartic incomplete $2$-balls that are curvature sharp at their center.

\begin{proposition}
There are 365 non-isomorphic quartic incomplete $2$-balls $B_2^\incomplete(v_0)$. Among them, there are 204 quartic incomplete $2$-balls that have non-negative curvature $\mathcal K_\infty(v_0)$. For more details, see Table \ref{table:count_2ball}.
\end{proposition}

\begin{table}[h!]
  \begin{center}
    \resizebox{\columnwidth}{!}{
		\begin{tabular}{| l | l | l | p{2.8cm} | p{2.8cm} |}
			\hline
			Index & $S_1$ structure & $S_1$-out regular & Number of $2$-balls & Number of $2$-balls with $\mathcal K_\infty(v_0) \ge 0$\\ \hline
			1 & [0, 0, 0, 0, 0, 0] & True & 93 & 46 \\ \hline
                        2 & [1, 0, 0, 0, 0, 0] & False & 120 & 55 \\ \hline
			3 & [1, 0, 0, 0, 0, 1] & True & 40 & 24 \\ \hline
			4 & [1, 1, 0, 0, 0, 0] & False & 55 & 31 \\ \hline
			5 & [1, 1, 1, 0, 0, 0] & False & 8 & 8 \\ \hline
			6 & [1, 1, 0, 1, 0, 0] & False & 10 & 4 \\ \hline
                        7 & [1, 1, 0, 0, 1, 0] & False & 24 & 21 \\ \hline
			8 & [1, 1, 0, 0, 1, 1] & True & 7 & 7 \\ \hline
			9 & [1, 1, 1, 1, 0, 0] & False & 5 & 5 \\ \hline
			10 & [1, 1, 1, 1, 1, 0] & False & 2 & 2 \\ \hline
			11 & [1, 1, 1, 1, 1, 1] & True & 1 & 1 \\ \hline\hline
                 	\multicolumn{3}{|c|}{Total} & 365 & 204 \\ \hline
		\end{tabular}}
            \end{center}
            \vspace*{.1cm}
	\caption{Number of incomplete $2$-balls classified by their $S_1$ structure}
	\label{table:count_2ball}
\end{table}

In fact the above result is not used in our proof in Section \ref{sect:proof_main}. The following result, however, is crucial for the proof:

\begin{proposition}
  There are 22 non-isomorphic quartic incomplete $2$-balls $B_2^\incomplete(v_0)$ which are curvature sharp in $v_0$. They are listed in Table \ref{table: cs}.
\end{proposition}

\begin{table}[h]
  \resizebox{\columnwidth}{!}{
  	\begin{tabular}{|c|l|l|l|c|}
		\hline
		$\#_\Delta(e),$ & \multicolumn{3}{ c| }{Incomplete $2$-ball $B_2^{\rm inc}(v_0)$} & $\mathcal K_\infty(v_0)$ \\ \cline{2-4}
		$v_0\in e\in E$ & Index & $S_1$ structure & $S_1$-$S_2$ structure & \\ \hline
		3 &1.1 & [1, 1, 1, 1, 1, 1] &$\emptyset$& $3.5$\\ \hline\hline
		
		&2.1 & [1, 1, 0, 0, 1, 1]& [1234]& \\ \cline{2-4}
		&2.2 & [1, 1, 0, 0, 1, 1]& [123], [4]& \\ \cline{2-4}
		&2.3 & [1, 1, 0, 0, 1, 1]& [12], [3], [4]& \\ \cline{2-4}
		2&2.4 & [1, 1, 0, 0, 1, 1]& [14], [2], [3]& $3.0$\\ \cline{2-4}
		&2.5 & [1, 1, 0, 0, 1, 1]& [1], [2], [3], [4]& \\ \cline{2-4}
		&2.6 & [1, 1, 0, 0, 1, 1]& [12], [34]& \\ \cline{2-4}
		&2.7 & [1, 1, 0, 0, 1, 1]& [14], [23] &\\ \hline\hline
		
		&3.1 & [1, 0, 0, 0, 0, 1] & [1234], [13], [24]& \\ \cline{2-4}
		1&3.2 & [1, 0, 0, 0, 0, 1]& [13], [13], [24], [24]& $2.5$ \\ \cline{2-4}
		&3.3 & [1, 0, 0, 0, 0, 1]& [13], [14], [23], [24]& \\ \cline{2-4}
		&3.4 & [1, 0, 0, 0, 0, 1]& [1234], [1234]& \\ \hline\hline
		
		&4.1 & [[0, 0, 0, 0, 0, 0]& [1234], [1234], [1], [2], [3], [4] & \\ \cline{2-4}
		&4.2 & [0, 0, 0, 0, 0, 0]& [1234], [1234], [12], [3], [4]& \\ \cline{2-4}
		&4.3 & [0, 0, 0, 0, 0, 0]& [1234], [1234], [123], [4]& \\ \cline{2-4}
		&4.4 & [0, 0, 0, 0, 0, 0] & [1234], [12], [13], [24], [34]& \\ \cline{2-4}
		0&4.5 & [0, 0, 0, 0, 0, 0] & [12], [13], [14], [23], [24], [34]& $2.0$\\ \cline{2-4}
		&4.6 & [0, 0, 0, 0, 0, 0] & [123], [123], [14], [24], [34]& \\ \cline{2-4}
		&4.7 & [0, 0, 0, 0, 0, 0] & [1234], [1234], [12], [34]& \\ \cline{2-4}
		&4.8 & [0, 0, 0, 0, 0, 0] & [1234], [123], [124], [34]& \\ \cline{2-4}
		&4.9 & [0, 0, 0, 0, 0, 0] & [123], [124], [134], [234]&  \\ \cline{2-4}
		&4.10 & [0, 0, 0, 0, 0, 0] & [1234], [1234], [1234]&  \\ \hline
	\end{tabular}}
      \vspace*{.1cm}
	\caption{Incomplete $2$-ball structures with a curvature sharp center}
	\label{table: cs}
\end{table}

\section{Proof of the classification theorem} \label{sect:proof_main}

\begin{proof}[Proof of Theorem \ref{thm:main}]
  Let us start with a reference vertex $v_0$ and
  $S_1(v_0) = \{v_1,v_2,v_3,v_4\}$. We will perform a case-by-case
  analysis of all 22 possible (non-isomorphic)
  $B_2^{\rm inc}(v_0)$-structures which are curvature sharp at $v_0$,
  provided in Table \ref{table: cs}.

  Since we assume all vertices $v$ to be curvature sharp, any
  incomplete $2$-ball $B_2^{\rm inc}(v)$ must be one of the 22
  possible types, but their types can differ from vertex to
  vertex. However, we will see a posteriori that each globally
  curvature sharp graph generated by these cases is vertex transitive
  and, therefore, the incomplete $2$-ball types of all its vertices
  coincide.

Moreover, Proposition \ref{prop: number_triangles} asserts that for every edge $e\in E$ the number of triangles containing $e$, $\#_\Delta(e)$, is uniform. We will therefore use the number $\#_\Delta(e)\in\{0,1,2,3\}$ for our case separation.
Table \ref{table: cs_graphs} provides an overview about all incomplete $2$-ball structures that lead to globally curvature sharp graphs.

\begin{table}[h]
    \resizebox{\columnwidth}{!}{
	\begin{tabular}{|c|l|l|l|c|}
		\hline
		$\#_\Delta(e),$ & \multicolumn{3}{ c| }{Incomplete $2$-ball $B_2^{\rm inc}(v_0)$} & resulting graph(s) \\ \cline{2-4}
		$v_0\in e\in E$ & Index & $S_1$ structure & $S_1$-$S_2$ structure & \\ \hline
		3 &1.1 & [1, 1, 1, 1, 1, 1] &$\emptyset$& $K_5$\\ \hline
		
		2&2.1 & [1, 1, 0, 0, 1, 1]& [1234]& $O$\\ \cline{1-5}
		
		1&3.3 & [1, 0, 0, 0, 0, 1]& [13], [14], [23], [24]& $K_3\times K_3$\\ \cline{1-5}
		
		&4.5 & [0, 0, 0, 0, 0, 0] & [12], [13], [14], [23], [24], [34]& $Cay(D_{14},S)$ and $Q^4$\\ \cline{2-5}
		0&4.6 & [0, 0, 0, 0, 0, 0] & [123], [123], [14], [24], [34]& $Cay(D_{12},S)$ \\ \cline{2-5}
		&4.9 & [0, 0, 0, 0, 0, 0] & [123], [124], [134], [234]& $C(10)$ \\ \cline{2-5}
		&4.10 & [0, 0, 0, 0, 0, 0] & [1234], [1234], [1234]& $K_{4,4}$ \\ \hline
	\end{tabular}}
      \vspace*{.1cm}
	\caption{Incomplete $2$-ball structures leading to globally curvature sharp graphs}
	\label{table: cs_graphs}
\end{table}

\subsection{Case $\#_\Delta(e)=3$ or $\#_\Delta(e)=2$}

In the case $\#_\Delta(e)=3$, the $S_1$-structure immediately implies that $G$ is the complete graph $K_5$.

Next we deal with the case $\#_\Delta(e)=2$. If $B_2^{\rm inc}(v_0)$ is of type 2.1, then $G$ is immediately the Octahedral graph $O$. Otherwise, if $B_2^{\rm inc}(v_0)$ is of type 2.2, 2.3, 2.4, 2.5, 2.6 or 2.7, the $S_1$-$S_2$ structure infers that $v_2$ and $v_4$ have no common neighbor in $S_2(v_0)$. Thus we have $\#_\Delta(\{v_2,v_4\})=1$, namely the triangle $\{v_0v_2v_4\}$; contradiction to $\#_\Delta(e)=2$.

\subsection{Case $\#_\Delta(e)=1$}

In the case $\#_\Delta(e)=1$, we will show that the incomplete $2$-ball $B_2^{\rm inc}(v_0)$ has only one possible structure, which is of type 3.3. Moreover, it leads to a unique graph $G$, namely the Cartesian product $K_3\times K_3$.

\subsubsection{Case $B_2^{\rm inc}(v_0)$ is of type 3.1} Denote $v_5,v_6,v_7\in S_2(v_0)$ with the patterns $$v_5\equiv [1234] \qquad v_6\equiv [13] \qquad v_7\equiv [24].$$ Then $\{v_0v_1v_2\}$ and $\{v_1v_2v_5\}$ are triangles, so $\#_\Delta(\{v_1,v_2\})\ge 2$; contradiction.

We purposedly use ``$\equiv$'' to describe the pattern of a vertex in a $2$-sphere to allow the possibility that two vertices may have the same pattern e.g., $v_5\equiv v_6\equiv [13]$ even though $v_5\not=v_6$.

\subsubsection{Case $B_2^{\rm inc}(v_0)$ is of type 3.2} Denote $v_5,v_6,v_7,v_8\in S_2(v_0)$ with the patterns $$v_5\equiv v_6\equiv [13] \qquad v_7\equiv v_8\equiv [24].$$ Note that $v_1$ has four neighbors $v_0$, $v_2$, $v_5$, $v_6$. Since $v_5$ is not a neighbor of $v_0$ and $v_2$, the fact that $\#_\Delta(\{v_1,v_5\})=1$ implies that $v_5$ is a neighbor of $v_6$. Now $\{v_1v_5v_6\}$ and $\{v_3v_5v_6\}$ are triangles, so $\#_\Delta(\{v_5,v_6\})\ge 2$; contradiction.

\subsubsection{Case $B_2^{\rm inc}(v_0)$ is of type 3.4} Denote $v_5,v_6\in S_2(v_0)$ with the patterns $v_5\equiv v_6\equiv [1234]$. Then $\{v_0v_1v_2\}$, $\{v_1v_2v_5\}$, and $\{v_1v_2v_6\}$ are triangles, so $\#_\Delta(\{v_1,v_2\})=3$; contradiction.

\subsubsection{Case $B_2^{\rm inc}(v_0)$ is of type 3.3 } Denote $v_5,v_6,v_7,v_8\in S_2(v_0)$ with the patterns $$v_5\equiv [13] \quad v_6\equiv [14] \quad v_7\equiv [23] \quad v_8\equiv [24].$$ Consider $v_1$ as a center with four neighbors $v_0$, $v_2$, $v_5$, $v_6$. Since $v_5\not\sim v_0$ and $v_5\not\sim v_2$, the fact that $\#_\Delta(\{v_1,v_5\})=1$ implies $v_5\sim v_6$. Similarly,\\
by centering at $v_2$, $\#_\Delta(\{v_2,v_7\})=1$ implies $v_7\sim v_8$.\\ By centering at $v_3$, $\#_\Delta(\{v_3,v_5\})=1$ implies $v_5\sim v_7$.\\ By centering at $v_4$, $\#_\Delta(\{v_4,v_6\})=1$ implies $v_6\sim v_8$.

Now $B_2^{\rm inc}(v_0)$ with additional edges $v_5\sim v_6 \sim v_8 \sim v_7 \sim v_5$ results in a quartic graph, which is in fact the Cartesian product $K_3\times K_3$ (see Figure \ref{fig:cart K_3}).

\begin{figure}[h!]
	\begin{center}
		\begin{tikzpicture} [scale=0.4]
		\node[above] at (0,0) {$v_0$}; \node[above] at (4,3) {$v_1$}; \node[below] at (4,-3) {$v_2$}; \node[above] at (6,0.9) {$v_3$}; \node[above right] at (10,3.9) {$v_5$}; \node[above right] at (10,-2.1) {$v_7$}; \node[above] at (12,0) {$v_4$}; \node[above] at (16,3) {$v_6$}; \node[below] at (16,-3) {$v_8$};	 
		
		\draw (0,0)--(4,3)--(4,-3)--(0,0);
		\draw (6,0.9)--(10,3.9)--(10,-2.1)--(6,0.9);
		\draw (12,0)--(16,3)--(16,-3)--(12,0);
		
		\draw (0,0)--(6,0.9)--(12,0)--(0,0);
		\draw (4,3)--(10,3.9)--(16,3)--(4,3);
		\draw (4,-3)--(10,-2.1)--(16,-3)--(4,-3);
		\end{tikzpicture}
		
	\end{center}
	\caption{Cartesian product $K_3\times K_3$}
	\label{fig:cart K_3}
\end{figure}
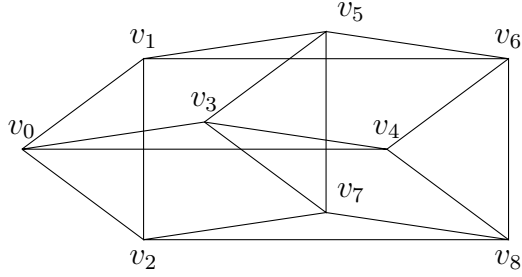

\subsection{Case $\#_\Delta(e)=0$}
Lastly, we deal with the most difficult case where $\#_\Delta(e)=0$ (i.e. $G$ is triangle-free) and we expect to derive $5$ possibilities of $G$, depending on its incomplete $2$-ball structure. Henceforth, we restrict ourselves to the ``bottom half" of Table \ref{table: cs}, that is the ones indexed by 4.1-4.10.

From now on, we introduce a new notation for $S_1$-$S_2$ structure of the $2$-ball $B_2^{\rm inc}(v_i)$, which is centered around the vertex $v_i$ (for $i\in\{1,2,3,4\}$). To do so, we add a subscript $i$ to the pattern of each vertex on the two-sphere $S_2(v_i)$.

For example, assuming $v_1$ has the neighbors $v_0,v_5,v_6,v_7$, the subscripts $1$ in the following patterns
\begin{align*} 
v_2\equiv [0567]_1 \qquad v_3\equiv [0567]_1 \qquad v_4\equiv [056]_1
\end{align*}
signifies that they describe vertices in the $S_1$-$S_2$ structure of $B_2^{\rm inc}(v_1)$.

Moreover, as in previous arguments, patterns are written without subscripts when we describe the $S_1$-$S_2$ structure of $B_2^{\rm inc}(v_0)$.

\subsubsection{Case $B_2^{\rm inc}(v_0)$ is of type 4.10}
It is straightforward to deduce that $G$ is the complete bipartite graph $K_{4,4}$.

\subsubsection{Case $B_2^{\rm inc}(v_0)$ is of type 4.1, 4.2, 4.3, or 4.7}

Note that all these cases have precisely two vertices in $S_2(v_0)$ with the patterns $[1234]$. Let us denote them by $v_5$ and $v_6$.

Consider $v_1$ as a center with three known neighbors $v_0,v_5,v_6$, and the other unknown neighbor, which we denote by $v_7$. Now $v_2,v_3,v_4$ are in $S_2(v_1)$ and all of them are neighbors of $v_0,v_5,v_6$, so they will have the following patterns:
\begin{align} \label{structure_1}
v_2\equiv [056*_2]_1 \qquad v_3\equiv [056*_3]_1 \qquad v_4\equiv [056*_4]_1
\end{align}
where (for $i\in\{2,3,4\}$) each $*_i$ takes value $7$ or ``empty", depending on whether $v_i$ is a neighbor of $v_7$ or not.

When comparing the patterns in \eqref{structure_1} to the $S_1$-$S_2$ structures in Table \ref{table: cs}, we can see that the possible structures of $B_2^{\rm inc}(v_1)$ are of type either 4.3 or 4.10. However, an incomplete $2$-ball of type 4.10 previously led to the resulting graph $G=K_{4,4}$. In case $B_2^{\rm inc}(v_1)$ is of type 4.3, two vertices have their patterns $[0567]_1$ and one vertex has its pattern $[056]_1$. Without loss of generality, let the patterns in \eqref{structure_1} take values 
\begin{align*} 
v_2\equiv [0567]_1 \qquad v_3\equiv [0567]_1 \qquad v_4\equiv [056]_1
\end{align*}
which means $v_4\not\sim v_7$. Now consider $v_4$ as a center, with neighbors $v_0,v_5,v_6$, and another neighbor denoted by $v_8 (\not=v_7)$. Note that due to the $S_1$ structure of $v_0$ being $[0,0,0,0,0,0]$, it means that $v_4$ is not a neighbor of $v_1,v_2,v_3$. Thus $v_1,v_2,v_3$ are in $S_2(v_4)$, and all of them are neighbors of $v_0,v_5,v_6,v_7$ (but not of $v_8$). Hence, $v_1,v_2,v_3$ have all the same pattern $[056]_4$, which does not belong to any $S_1$-$S_2$ structure in Table \ref{table: cs}.

\subsubsection{Case $B_2^{\rm inc}(v_0)$ is of type 4.4}
Denote the vertices on $S_2(v_0)$ by patterns
\begin{align*} 
v_5\equiv [1234] \qquad v_6\equiv [12] \qquad v_7\equiv [13] \qquad v_8\equiv [24] \qquad v_9\equiv [34].
\end{align*}
Consider $v_1$ as a center with four neighbors $v_0,v_5,v_6,v_7$. Now $v_2,v_3,v_4$ are in $S_2(v_1)$ with the patterns
\begin{align*} 
v_2\equiv [056]_1 \qquad v_3\equiv [057]_1 \qquad v_4\equiv [05]_1,
\end{align*}
which does not belong to any $S_1$-$S_2$ structure in Table \ref{table: cs}.

\subsubsection{Case $B_2^{\rm inc}(v_0)$ is of type 4.8}
Denote the vertices on $S_2(v_0)$ by patterns
\begin{align*} 
v_5\equiv [1234] \qquad v_6\equiv [123] \qquad v_7\equiv [124] \qquad v_8\equiv [34].
\end{align*}
Consider $v_3$ as a center with four neighbors $v_0,v_5,v_6,v_8$. Now $v_1,v_2,v_4$ are in $S_2(v_3)$ with the structure
\begin{align*} 
v_1\equiv [056]_3 \qquad v_2\equiv [056]_3 \qquad v_4\equiv [058]_3,
\end{align*}
which does not belong to any $S_1$-$S_2$ structure in Table \ref{table: cs}.

\subsubsection{Case $B_2^{\rm inc}(v_0)$ is of type 4.6} 
From now on, instead of calling the vertices on $S_2(v_0)$ as $v_5,v_6$ and so on, we call them differently by names reflecting their patterns.  

In this particular case, we will denote the vertices on $S_2(v_0)$ by patterns
\begin{align*} 
v_{123}\equiv v'_{123}\equiv [123] \qquad v_{14}\equiv [14] \qquad v_{24}\equiv [24] \qquad v_{34}\equiv [34].
\end{align*}

Consider $v_1$ as a center with four neighbors $v_0,v_{14},v_{123},v'_{123}$. Henceforth, we also describe patterns no longer just by the indices of the involved vertices but by the vertices themselves. In this case, the vertices $v_2,v_3,v_4$ are in $S_2(v_1)$ with patterns
\begin{align*} 
v_2\equiv v_3\equiv [v_0v_{123}v'_{123}]_1 \qquad v_4\equiv [v_0v_{14}]_1,
\end{align*}
and according to Table \ref{table: cs}, the $B_2^{\rm inc}(v_1)$ must be of type 4.6. That is, the other two vertices in $S_2(v_1)$, namely $A$ and $B$, will have patterns
\begin{align} \label{structure:v14_1} 
A\equiv [v_{123}v_{14}]_1 \qquad B\equiv [v'_{123}v_{14}]_1
\end{align}
In principle, it is possible that $A$ or $B$ could coincide with $v_{24}$ or $v_{34}$. However, this can be excluded by the following arguments. If $A=v_{24}$, we would have a triangle $\{v_4v_{14}v_{24}\}$, and if $A=v_{34}$, we would have a triangle $\{v_4v_{14}v_{34}\}$. The same reasoning applies to $B=v_{24}$ or $B=v_{34}$.

Next, consider $v_4$ as a center with four neighbors $v_0,v_{14},v_{24},v_{34}$. Now $v_1,v_2,v_3,A,B$ are in $S_2(v_4)$ with the patterns
\begin{align*} 
v_1\equiv [v_0v_{14}]_4 \qquad v_2\equiv [v_0v_{24}]_4 \qquad v_3\equiv [v_0v_{34}]_4 \qquad A\equiv [v_{14}*]_4 \qquad B\equiv [v_{14}*]_4,
\end{align*}
where each $*$ represents some unknown vertex/vertices. According to Table \ref{table: cs}, the only possible type of $B_2^{\rm inc}(v_4)$ is 4.6. That is, $A$ and $B$ are in $S_2(v_4)$ with patterns
\begin{align} \label{structure:v14_2}  
A\equiv B\equiv [v_{14}v_{24}v_{34}]_4.
\end{align}

The information \eqref{structure:v14_1} and \eqref{structure:v14_2}
tells us that we have a quartic graph as in Figure
\ref{fig:cayley_12}, which is indeed a Cayley graph of $D_{12}$ (see
Figure \ref{fig:cayley_12_proof} in Remark \ref{remark: cayley}
below).

\begin{figure}[h!]
\begin{center}
\begin{tikzpicture} [scale=0.4]
\node[below] at (0,0) {$v_0$}; \node[below] at (4,3) {$v_1$}; \node[below] at (4,1) {$v_2$}; \node[below] at (4,-1) {$v_3$}; \node[below] at (4,-3) {$v_4$};
\draw (0,0)--(4,3); \draw (0,0)--(4,1); \draw (0,0)--(4,-1); \draw (0,0)--(4,-3);

\node[below] at (8,4) {$v_{123}$}; \node[below] at (8,2) {$v'_{123}$}; \node[below] at (8,0) {$v_{14}$}; \node[below] at (8,-2) {$v_{24}$}; \node[below] at (8,-4) {$v_{34}$};
\draw (4,3)--(8,4); \draw (4,3)--(8,2); \draw (4,3)--(8,0); 
\draw (4,1)--(8,4); \draw (4,1)--(8,2); \draw (4,1)--(8,-2);
\draw (4,-1)--(8,4); \draw (4,-1)--(8,2); \draw (4,-1)--(8,-4);
\draw (4,-3)--(8,0); \draw (4,-3)--(8,-2); \draw (4,-3)--(8,-4);

\node[below] at (12,2) {$A$}; \node[below] at (12,-2) {$B$};
\draw (8,4)--(12,2); \draw (8,0)--(12,2); \draw (8,-2)--(12,2); \draw (8,-4)--(12,2);
\draw (8,2)--(12,-2); \draw (8,0)--(12,-2); \draw (8,-2)--(12,-2); \draw (8,-4)--(12,-2);
\end{tikzpicture}
\end{center}
\caption{The unique graph arising in case $B_2^{\rm inc}(v_0)$ is of type 4.6}
\label{fig:cayley_12}
\end{figure}
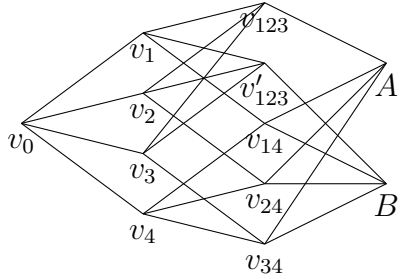

\subsubsection{Case $B_2^{\rm inc}(v_0)$ is of type 4.5} 
Denote the vertices on $S_2(v_0)$ by patterns
\begin{align*} 
v_{12}\equiv [12] \quad v_{13}\equiv [13] \quad v_{14}\equiv [14] \quad v_{23}\equiv [23] \quad v_{24}\equiv [24] \quad v_{34}\equiv [34].
\end{align*}

For each $i\in \{1,2,3,4\}$, consider $v_i$ as a center with four neighbors $v_0,v_{ij},v_{ik},v_{il}$ where $\{i,j,k,l\}=\{1,2,3,4\}$ (and from now on, $v_{ij}=v_{ji}$ by convention). The vertices $v_j,v_k,v_l$ are then in $S_2(v_i)$ with patterns

\begin{align*} 
v_j\equiv [v_0v_{ij}]_i \qquad v_k\equiv [v_0v_{ik}]_i  \qquad v_l\equiv [v_0v_{il}]_i.
\end{align*}

According to Table \ref{table: cs}, the type of $B_2^{\rm inc}(v_i)$ must be either 4.5 or 4.6 (and we can safely assume the type 4.5, since we previously dealt with the case where an incomplete $2$-ball is of type 4.6). 

Since $B_2^{\rm inc}(v_i)$ is of type 4.5, we suppose $S_2(v_i)=\{v_j,v_k,v_l,A_{jk}^i,A_{jl}^i,A_{kl}^i\}$ with patterns
\begin{align} \label{structure:case4.5_1}
v_j\equiv[v_0v_{ij}]_i \qquad v_k&\equiv[v_0v_{ik}]_i \qquad v_l\equiv[v_0v_{il}]_i \nonumber\\
A_{jk}^i\equiv [v_{ij}v_{ik}]_i \qquad A_{jl}^i&\equiv [v_{ij}v_{il}]_i \qquad A_{kl}^i\equiv [v_{ik}v_{il}]_i.
\end{align}
Here, $A_{jk}^i$ and $A_{kj}^i$ represent the same vertex. 

Note that none of the vertices $A_{jk}^i,A_{jl}^i,A_{kl}^i$ can coincide with the vertices $v_{jk},v_{jl},v_{kl}$, since this would always lead to the existence of some triangle, namely, $v_sv_{is}v_{st}$ for some distinct $s,t\in \{j,k,l\}$. In conclusion, $$A_{jk}^i,A_{jl}^i,A_{kl}^i \in S_3(v_0).$$

For every vertex $w\in S_3(v_0)$, we know that
\begin{equation} \label{eq:wvij}
w \sim v_{ij} \quad \text{for some distinct $i,j\in \{1,2,3,4\}$},
\end{equation}
and therefore $w\in S_2(v_i)\cap S_2(v_j)$. Since $w\in S_2(v_i)$, \eqref{structure:case4.5_1} implies that $w$ coincides with one of
$A_{jk}^i,A_{jl}^i,A_{kl}^i$, so it is adjacent to $v_{ij}$ by \eqref{eq:wvij} and
$v_{is}$ for some $s\not=j$.

Similarly, since $w\in S_2(v_j)$, $w$ is
also adjacent to $v_{ij}$ and $v_{jt}$ for some $t\not=i$. Therefore,
$w$ has at least $3$ different neighbors in $S_2(v_0)$, namely
$v_{ij}, v_{is}, v_{jt}$, so its in-degree (w.r.t. $v_0$)
is $$d^{-}_{v_0}(w)\ge 3 \qquad \mbox{ for all } w\in S_3(v_0).$$ On
the other hand, for every vertex $z\in S_2(v_0)$, its in-degree is
$d^{-}_{v_0}(z)= 2$, so its out-degree is $d^{+}_{v_0}(z)\le
2$. Counting edges between $S_2(v_0)$ and $S_3(v_0)$ then gives
\begin{align} \label{eq:count2-3} 
12 \ge \sum\limits_{z\in S_2(v_0)} d^{+}_{v_0}(z) = \sum\limits_{w\in S_3(v_0)} d^{-}_{v_0}(w) \ge 3|S_3(v_0)|,
\end{align}
so we have $|S_3(v_0)|\le 4$. Moreover, since $A_{jk}^i,A_{jl}^i,A_{kl}^i \in S_3(v_0)$, we must have $3\le |S_3(v_0)|\le 4$.

%
%
%
%

\begin{itemize}
\item \underline{Case $|S_3(v_0)|=3$}: then we know that
$S_3(v_0)=\{A_{jk}^i,A_{jl}^i,A_{kl}^i\}$. Observe also that $v_{ij}$ has four neighbors, namely $S_1(v_{ij})=\{v_i,v_j, A_{jk}^i,A_{jl}^i\}.$

Since these arguments hold for all indices $i,j$, we are allowed to interchange them, that is 
\begin{align*}
S_3(v_0)&=\{A_{jk}^i,A_{jl}^i,A_{kl}^i\}=\{A_{ik}^j,A_{il}^j,A_{kl}^j\}\\
S_1(v_{ij})&=\{v_i,v_j, A_{jk}^i,A_{jl}^i\}=\{v_i,v_j, A_{ik}^j,A_{il}^j\}.
\end{align*}
which implies that $A_{kl}^i$ and $A_{kl}^j$ coincide. By definition, it means that this vertex $A_{kl}^i\in S_3(v_0)$ is connected to $v_{ik},v_{il},v_{jk},v_{jl}$,  that is $$S_1(A_{kl}^i)=\{v_{ik},v_{il},v_{jk},v_{jl}\}.$$ Since this argument holds for all combinations of $i,j,k,l$, we derive the information about all three vertices $A_{23}^1,A_{24}^1,A_{34}^1\in S_3(v_0)$: 
\begin{align*}
S_1(A_{23}^1)&=\{v_{12},v_{13},v_{24},v_{34}\},\\
S_1(A_{24}^1)&=\{v_{12},v_{14},v_{23},v_{34}\},\\
S_1(A_{34}^1)&=\{v_{13},v_{14},v_{23},v_{24}\},
\end{align*}
which results in a quartic graph as in Figure \ref{fig:cayley_14},
which is indeed a Cayley graph of $D_{14}$ (see Figure
\ref{fig:cayley_14_proof} in Remark \ref{remark: cayley}).

\begin{figure}[h!]
	\begin{center}
		\begin{tikzpicture} [scale=0.4]
		\node[below] at (0,0) {$v_0$}; \node[below] at (4,3) {$v_1$}; \node[below] at (4,1) {$v_2$}; \node[below] at (4,-1) {$v_3$}; \node[below] at (4,-3) {$v_4$};
		\draw (0,0)--(4,3); \draw (0,0)--(4,1); \draw (0,0)--(4,-1); \draw (0,0)--(4,-3);
		
		\node[below] at (8,5) {$v_{12}$}; \node[below] at (8,3) {$v_{13}$}; \node[below] at (8,1) {$v_{14}$}; \node[below] at (8,-1) {$v_{23}$}; \node[below] at (8,-3) {$v_{24}$}; \node[below] at (8,-5) {$v_{34}$};
		
		\draw (4,3)--(8,5); \draw (4,3)--(8,3); \draw (4,3)--(8,1); 
		\draw (4,1)--(8,5); \draw (4,1)--(8,-1); \draw (4,1)--(8,-3);
		\draw (4,-1)--(8,3); \draw (4,-1)--(8,-1); \draw (4,-1)--(8,-5);
		\draw (4,-3)--(8,1); \draw (4,-3)--(8,-3); \draw (4,-3)--(8,-5);
		
		\node[right] at (12,2) {$A_{23}^1$}; \node[right] at (12,0) {$A_{24}^1$}; \node[right] at (12,-2) {$A_{34}^1$};
		
		\draw (8,5)--(12,2); \draw (8,5)--(12,0);
		\draw (8,3)--(12,2); \draw (8,3)--(12,-2);
		\draw (8,1)--(12,0); \draw (8,1)--(12,-2);
		\draw (8,-1)--(12,0); \draw (8,-1)--(12,-2);
		\draw (8,-3)--(12,2); \draw (8,-3)--(12,-2);
		\draw (8,-5)--(12,2); \draw (8,-5)--(12,0);	
		\end{tikzpicture}
	\end{center}
	\caption{The graph arising in case $B_2^{\rm inc}(v_0)$ is of type 4.6 and $|S_3(v_0)|=3$}
	\label{fig:cayley_14}
\end{figure}
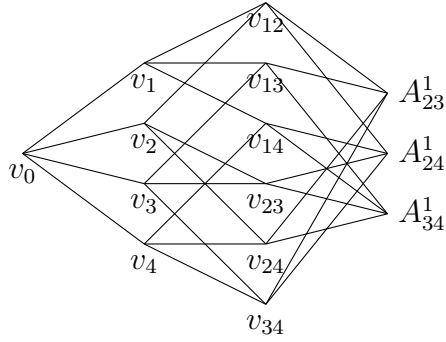

\item \underline{Case $|S_3(v_0)|=4$}: We claim that in this case our graph $G$ has diameter $\diam(G)\ge 4$. Assume for the sake of contradiction that $\diam(G)=3$ and the 4-sphere $S_4(v_0)$ is empty.

Since $|S_3(v_0)|=4$, the inequality \eqref{eq:count2-3} holds with equality, which implies that for all $w\in S_3(v_0)$, the in-degree $d^{-}_{v_0}(w)=3$. Thus the spherical degree $d^{0}_{v_0}(w)=1$. In particular, each of the vertices $A_{23}^1,A_{24}^1,A_{34}^1$ must be adjacent to another vertex in $S_3(v_0)$. However, note that no pair of vertices $A_{23}^1,A_{24}^1,A_{34}^1$ are adjacent (otherwise, if $A_{jk}^1$ and $A_{jl}^1$ were connected, then they would form a triangle $\Delta v_{1j}A_{jk}^1A_{jl}^1$). Thus $A_{23}^1,A_{24}^1,A_{34}^1$ must be adjacent to the the fourth vertex in $S_3(v_0)$, which we may denote by $P$. Then the spherical degree $d^{0}_{v_0}(P)=3$, contradiction.

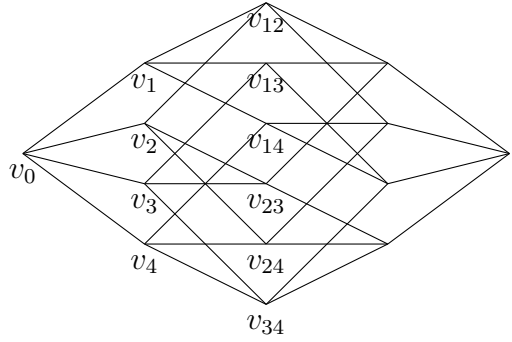
\begin{figure}[h!]
	\begin{center}
		\begin{tikzpicture} [scale=0.4]
		\node[below] at (0,0) {$v_0$}; \node[below] at (4,3) {$v_1$}; \node[below] at (4,1) {$v_2$}; \node[below] at (4,-1) {$v_3$}; \node[below] at (4,-3) {$v_4$};
		\draw (0,0)--(4,3); \draw (0,0)--(4,1); \draw (0,0)--(4,-1); \draw (0,0)--(4,-3);
		
		\node[below] at (8,5) {$v_{12}$}; \node[below] at (8,3) {$v_{13}$}; \node[below] at (8,1) {$v_{14}$}; \node[below] at (8,-1) {$v_{23}$}; \node[below] at (8,-3) {$v_{24}$}; \node[below] at (8,-5) {$v_{34}$};
		
		\draw (4,3)--(8,5); \draw (4,3)--(8,3); \draw (4,3)--(8,1); 
		\draw (4,1)--(8,5); \draw (4,1)--(8,-1); \draw (4,1)--(8,-3);
		\draw (4,-1)--(8,3); \draw (4,-1)--(8,-1); \draw (4,-1)--(8,-5);
		\draw (4,-3)--(8,1); \draw (4,-3)--(8,-3); \draw (4,-3)--(8,-5);
		
		
		\draw (8,5)--(12,3); \draw (8,5)--(12,1);
		\draw (8,3)--(12,3); \draw (8,3)--(12,-1);
		\draw (8,1)--(12,1); \draw (8,1)--(12,-1);
		\draw (8,-1)--(12,3); \draw (8,-1)--(12,-3);
		\draw (8,-3)--(12,1); \draw (8,-3)--(12,-3);
		\draw (8,-5)--(12,-1); \draw (8,-5)--(12,-3);
		
		\draw (12,3)--(16,0); \draw (12,1)--(16,0); 
		\draw (12,-1)--(16,0); \draw (12,-3)--(16,0);
			
		\end{tikzpicture}
	\end{center}
	\caption{The 4-dimensional hypercube $Q^4$}
	\label{fig:Q^4}
\end{figure}

Therefore we have shown that $\diam(G) \ge 4$. On the other hand, Theorem \ref{thm:BMrigidity} gives the following diameter bound for Bakry-\'Emery curvature: $$\diam(G)\le \frac{2D}{K}=\frac{2\cdot4}{2}=4,$$ as we are working with quartic graphs ($D=4$) and $K=\inf\limits_{x} \mathcal(K)_\infty(x)=2$ (see Table \ref{table: cs}). Since the graph $G$ has diameter $\diam(G)=4$, it must be the $4$-dimensional hypercube $Q^4$ (as illustrated in Figure \ref{fig:Q^4}) by the rigidity statement in Theorem \ref{thm:BMrigidity}.

\end{itemize}

\subsubsection{Case $B_2^{\rm inc}(v_0)$ is of type 4.9} 
Denote the vertices on $S_2(v_0)$ by patterns
\begin{align*} 
v_{123}\equiv[123] \quad\quad v_{124}\equiv[124] \quad\quad v_{134}\equiv[134] \quad\quad v_{234}\equiv[234].
\end{align*}

For any $\{i,j,k,l\}=\{1,2,3,4\}$, each vertex $v_{ijk}\in S_2(v_0)$
has in-degree $d^{-}_{v_0}(v_{ijk})=3$ and spherical-degree
$d^{0}_{v_0}(v_{ijk})=0$ (as $G$ is triangle-free), so $v_{ijk}$ must
be connected to exactly one vertex on $S_3(v_0)$ (depending on the choice of $i,j,k$) which we denote by
$A_{ijk} \in S_3(v_0)$. Note that these vertices
$A_{ijk}$ might coincide (in fact, we will see that they are all the
same vertex).

Consider $v_i$ as a center with four neighbors $v_0,v_{ijk},v_{ijl},v_{ikl}$. The vertices $v_j,v_k,v_l\in S_2(v_i)$ will have patterns
\begin{align*}
v_j\equiv[v_0v_{ijk}v_{ijl}]_i \qquad v_k\equiv[v_0v_{ijk}v_{ikl}]_i \qquad v_l\equiv[v_0v_{ijl}v_{ikl}]_i.
\end{align*}

\begin{figure}[h!]
	\begin{center}
		\begin{tikzpicture} [scale=0.6]
		\node[below] at (0,0) {$v_0$}; \node[below] at (4,3) {$v_1$}; \node[below] at (4,1) {$v_2$}; \node[below] at (4,-1) {$v_3$}; \node[below] at (4,-3) {$v_4$};
		\draw (0,0)--(4,3); \draw (0,0)--(4,1); \draw (0,0)--(4,-1); \draw (0,0)--(4,-3);
		
	    \node[below] at (8,3) {$v_{123}$}; \node[below] at (8,1) {$v_{124}$}; \node[below] at (8,-1) {$v_{134}$}; \node[below] at (8,-3) {$v_{234}$};
		
		\draw (4,3)--(8,3); \draw (4,3)--(8,1); \draw (4,3)--(8,-1); 
		\draw (4,1)--(8,3); \draw (4,1)--(8,1); \draw (4,1)--(8,-3);
		\draw (4,-1)--(8,3); \draw (4,-1)--(8,-1); \draw (4,-1)--(8,-3);
		\draw (4,-3)--(8,1); \draw (4,-3)--(8,-1); \draw (4,-3)--(8,-3);
		
		\node[below] at (12,0) {$A$};
		
		\draw (8,3)--(12,0); \draw (8,1)--(12,0); \draw (8,-1)--(12,0); \draw (8,-3)--(12,0);		
		\end{tikzpicture}
	\end{center}
	\caption{The crown graph $C(10)$}
	\label{fig:crown}
\end{figure}
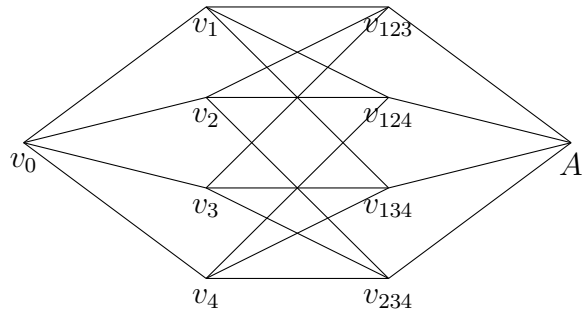

According to Table \ref{table: cs}, $B_2^{\rm inc}(v_i)$ must be of
type 4.9 and, therefore, $|S_2(v_i)|=4$. We denote the so far
unlabeled vertex of $S_2(v_i)$ by $A_i$ and we have $S_2(v_i) = \{v_j,v_k,v_l,A_i\}$ and $A_i$ will have a pattern
\begin{equation}\label{eq:A_i}
A_i\equiv[v_{ijk}v_{ijl}v_{ikl}]_i.
\end{equation}
Since $d(A_i,v_i)=2$ and $A_j \neq v_j,v_k,v_l$, we have
$$ A_i \not\in \{ v_0,v_i,v_j,v_k,v_{ijk},v_{ijl},v_{ikl} \}. $$
We can also rule out $A_i = v_{jkl}$ (for, otherwise,
$2=d(A_i,v_i)=d(v_{jkl},v_i)$ would imply that $A_i=v_{jkl}$ is
adjacent to one of the neighbors $v_{ijk},v_{ijl},v_{ikl}$ of $v_i$;
but any edge between two vertices of $v_{ijk},v_{ijl},v_{ikl},v_{jkl}$
would create a triangle and, therefore, a contradiction). These considerations
show that $A_i \in S_3(v_0)$ and $A_i \sim v_{ijk}$ by \eqref{eq:A_i}. By definition of the vertices $A_{ijk}$, we conclude that
$$A_i=A_{ijk} \qquad \mbox{for all permutations } \{i,j,k,l\}=\{1,2,3,4\}.$$
In particular,
\begin{align*}
A_1&=A_{123}=A_{124}=A_{134}\\
A_2&=A_{123}=A_{124}\qquad \qquad = A_{234},
\end{align*}
which means $A_{123},...,A_{234}$ all coincide, and $S_3(v_0)$ has only one vertex. As a result, $G$ is the crown graph $C(10)$, as shown in Figure \ref{fig:crown}.

After consideration of all possible cases we have now completed our classification result of quartic curvature sharp graphs (8 of them in total).
\end{proof}

\begin{remark} \label{remark: cayley}
Figures \ref{fig:cayley_12_proof} and \ref{fig:cayley_14_proof} illustrate that the graphs in (vi) and (vii) of Theorem \ref{thm:main} have indeed the stated Cayley graph structure.

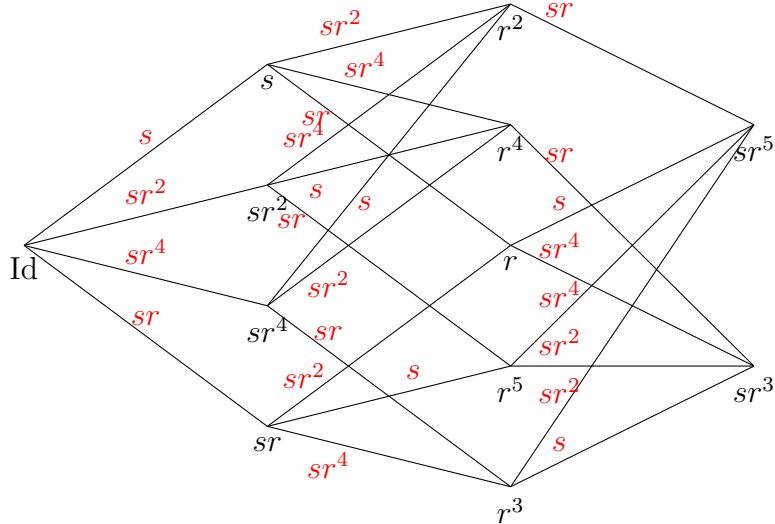
\begin{figure}[h!]
	\begin{center}
		\begin{tikzpicture} [scale=0.8]
		\node[below] at (0,0) {${\rm Id}$}; \node[below] at (4,3) {$s$}; \node[below] at (4,1) {$sr^2$}; \node[below] at (4,-1) {$sr^4$}; \node[below] at (4,-3) {$sr$};
                \draw (0,0)--(4,3) node [above,midway,red] {$s$};
		\draw (0,0)--(4,1) node [above,midway,red] {$sr^2$}; 
		\draw (0,0)--(4,-1) node [above,midway,red] {$sr^4$}; 
		\draw (0,0)--(4,-3) node [above,midway,red] {$sr$};
		
		\node[below] at (8,4) {$r^2$}; \node[below] at (8,2) {$r^4$}; \node[below] at (8,0) {$r$}; \node[below] at (8,-2) {$r^5$}; \node[below] at (8,-4) {$r^3$};
                \draw (4,3)--(8,4) node [above,pos=0.3,red] {$sr^2$};
		\draw (4,3)--(8,2) node [above,pos=0.4,red] {$sr^4$};
		\draw (4,3)--(8,0) node [below,pos=0.2,red] {$sr$};
		\draw (4,1)--(8,4) node [above,pos=0.15,red] {$sr^4$}; 
		\draw (4,1)--(8,2) node [below,pos=0.2,red] {$s$}; 
		\draw (4,1)--(8,-2) node [below,pos=0.1,red] {$sr$};
		\draw (4,-1)--(8,4) node [below,pos=0.4,red] {$s$}; 
		\draw (4,-1)--(8,2) node [below,pos=0.25,red] {$sr^2$}; 
		\draw (4,-1)--(8,-4) node [above,pos=0.25,red] {$sr$};
		\draw (4,-3)--(8,0) node [above,pos=0.15,red] {$sr^2$};
		\draw (4,-3)--(8,-2) node [above,pos=0.6,red] {$s$}; 
		\draw (4,-3)--(8,-4) node [below,pos=0.25,red] {$sr^4$};
		
		\node[below] at (12,2) {$sr^5$}; \node[below] at (12,-2) {$sr^3$};
		\draw (8,4)--(12,2) node [above,pos=0.2,red] {$sr$}; 
		\draw (8,0)--(12,2) node [above,pos=0.2,red] {$s$}; 
		\draw (8,-2)--(12,2) node [above,pos=0.2,red] {$sr^4$};
		\draw (8,-4)--(12,2) node [above,pos=0.2,red] {$sr^2$};
		\draw (8,2)--(12,-2) node [above,pos=0.2,red] {$sr$};
		\draw (8,0)--(12,-2) node [above,pos=0.2,red] {$sr^4$}; 
		\draw (8,-2)--(12,-2) node [above,pos=0.2,red] {$sr^2$}; 
		\draw (8,-4)--(12,-2) node [above,pos=0.2,red] {$s$};
		\end{tikzpicture}
	\end{center}
	\caption{The Cayley graph $Cay(D_{12},S)$ with $S=\{r^3,s,sr^2,sr^4\}$}
	\label{fig:cayley_12_proof}
\end{figure}

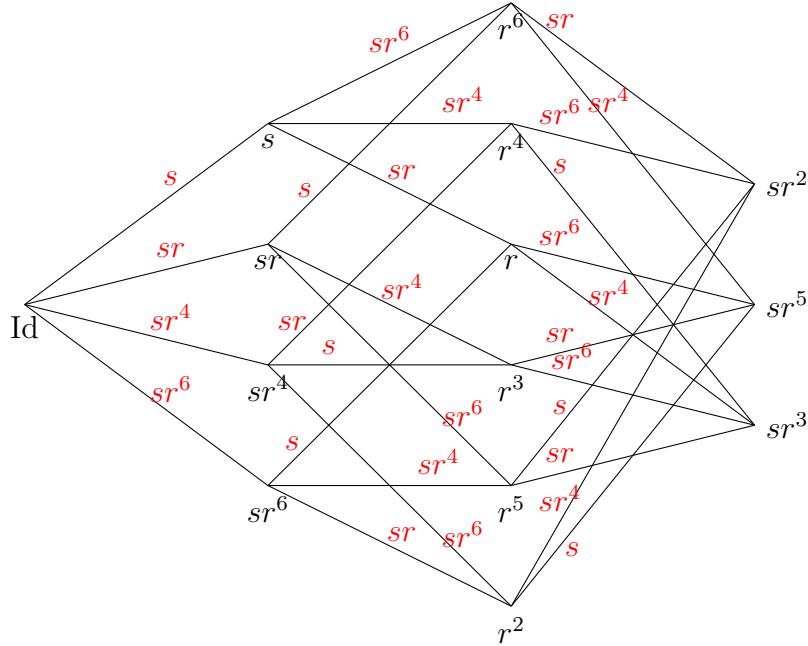
\begin{figure}[h!]
	\begin{center}
		\begin{tikzpicture} [scale=0.8]
		\node[below] at (0,0) {${\rm Id}$}; \node[below] at (4,3) {$s$}; \node[below] at (4,1) {$sr$}; \node[below] at (4,-1) {$sr^4$}; \node[below] at (4,-3) {$sr^6$};
		\draw (0,0)--(4,3)  node [above, pos=0.6,red] {$s$};
                \draw (0,0)--(4,1)  node [above, pos=0.6,red] {$sr$};
                \draw (0,0)--(4,-1) node [above, pos=0.6,red] {$sr^4$};
                \draw (0,0)--(4,-3) node [above, pos=0.6,red] {$sr^6$};
		
		\node[below] at (8,5) {$r^6$}; \node[below] at (8,3) {$r^4$}; \node[below] at (8,1) {$r$}; \node[below] at (8,-1) {$r^3$}; \node[below] at (8,-3) {$r^5$}; \node[below] at (8,-5) {$r^2$};
		
		\draw (4,3)--(8,5) node [above, pos=0.5,red] {$sr^6$};
                \draw (4,3)--(8,3) node [above, pos=0.8,red] {$sr^4$};
                \draw (4,3)--(8,1) node [above, pos=0.55,red] {$sr$}; 
		\draw (4,1)--(8,5) node [above, pos=0.15,red] {$s$};
                \draw (4,1)--(8,-1) node [above, pos=0.55,red] {$sr^4$};
                \draw (4,1)--(8,-3) node [above, pos=0.8,red] {$sr^6$};
		\draw (4,-1)--(8,3) node [above, pos=0.1,red] {$sr$};
                \draw (4,-1)--(8,-1) node [above, pos=0.25,red] {$s$};
                \draw (4,-1)--(8,-5) node [above, pos=0.8,red] {$sr^6$};
		\draw (4,-3)--(8,1) node [above, pos=0.1,red] {$s$};
                \draw (4,-3)--(8,-3) node [above, pos=0.7,red] {$sr^4$};
                \draw (4,-3)--(8,-5) node [above, pos=0.55,red] {$sr$};
		
		\node[right] at (12,2) {$sr^2$}; \node[right] at (12,0) {$sr^5$}; \node[right] at (12,-2) {$sr^3$};
		
		\draw (8,5)--(12,2) node [above, pos=0.2,red] {$sr$};
                \draw (8,5)--(12,0) node [above, pos=0.4,red] {$sr^4$};
		\draw (8,3)--(12,2) node [above, pos=0.2,red] {$sr^6$};
                \draw (8,3)--(12,-2) node [above, pos=0.2,red] {$s$};
		\draw (8,1)--(12,0) node [above, pos=0.2,red] {$sr^6$};
                \draw (8,1)--(12,-2) node [above, pos=0.4,red] {$sr^4$};
		\draw (8,-1)--(12,0) node [above, pos=0.2,red] {$sr$};
                \draw (8,-1)--(12,-2) node [above, pos=0.25,red] {$sr^6$};
		\draw (8,-3)--(12,2) node [above, pos=0.2,red] {$s$};
                \draw (8,-3)--(12,-2) node [above, pos=0.2,red] {$sr$};
		\draw (8,-5)--(12,2) node [above, pos=0.2,red] {$sr^4$};
                \draw (8,-5)--(12,0) node [below, pos=0.25,red] {$s$};	
		\end{tikzpicture}
	\end{center}
	\caption{The Cayley graph $Cay(D_{14},S)$ with $S = \{ s,sr,sr^4,sr^6 \}$}
	\label{fig:cayley_14_proof}
\end{figure}

\end{remark}

{\bf Acknowledgement:} Leyna Watson May's research was funded by the
London Mathematical Society (LMS) in a 2018 Undergraduate
Research Bursary. We thank Francis Gurr (Watson May's co-author
in \cite{GW18}) for his support in creating the relevant Python code
and\\ Riikka Kangaslampi for helpful mathematical discussions.



\begin{thebibliography}{9}
\bibitem[GW18]{GW18} F. Gurr, L. Watson May, {Bakry-\'Emery curvature
  calculations of quartic 2-balls and graphs of diameter 2}, Python code
  results used for our article are filed as ancillary file.
\end{thebibliography}
\end{document}